\documentclass[a4paper,leqno,final]{amsart}

\usepackage{ifdraft}

\ifdraft{\usepackage[draft]{showkeys}}{\usepackage[final]{showkeys}}

\usepackage[T1]{fontenc}
\usepackage[utf8]{inputenc}
\usepackage[final]{microtype}

\usepackage{amsmath,amsfonts,mathscinet,eucal,amsthm,amssymb,bm,extarrows,upgreek,tensor,mathrsfs,textcomp,comment,mathtools,bbm,braket}

\usepackage[shortlabels]{enumitem}

\usepackage[usenames,dvipsnames]{xcolor} \usepackage{tikz}
\usetikzlibrary{patterns,matrix,through,arrows,decorations.pathreplacing,decorations.markings,shadows,shapes.geometric,positioning,calc,backgrounds,fit}

\usetikzlibrary{external}
\tikzsetexternalprefix{cat-figures/}

\def \myweightstyle {}
\def \myweightx {0.09}
\def \myweighty {0.2}
\newcommand{\mydrawdown}[1]%
  {\draw [\myweightstyle] #1 -- ++(\myweightx,\myweighty);%
    \draw [\myweightstyle] #1 -- ++(-\myweightx,\myweighty);}
\newcommand{\mydrawup}[1]%
  {\draw [\myweightstyle] #1 -- ++(-\myweightx,-\myweighty);%
    \draw [\myweightstyle] #1 -- ++(\myweightx,-\myweighty);}

\tikzset{anchorbase/.style={baseline={([yshift=-0.5ex]current bounding box.center)}},
  int/.style={thick},
  cross line/.style={preaction={draw=white,line width=6pt,-}},
  wall/.style={thin,double,blue},
  middlearrow/.style={postaction=decorate,decoration={markings,mark=at
    position .55 with {\arrow{stealth};}}},
  middlearrowrev/.style={postaction=decorate,decoration={markings,mark=at
    position .55 with {\arrowreversed{stealth};}}},
  ev/.style={shape=rectangle, draw}
}

\usepackage{todonotes} 

\usepackage[bbgreekl]{mathbbol}

\DeclareSymbolFontAlphabet{\mathbb}{AMSb}
\DeclareSymbolFontAlphabet{\mathbbol}{bbold}
\DeclareMathAlphabet{\mathpzc}{OT1}{pzc}{m}{it}

\DeclareSymbolFont{usualmathcal}{OMS}{cmsy}{m}{n}
\DeclareSymbolFontAlphabet{\mathucal}{usualmathcal}

\numberwithin{equation}{section}

\newtheoremstyle{myplain} {6pt plus 6pt minus 2pt}
{6pt plus 6pt minus 2pt}
{\itshape}
{}
{\bfseries}
{.}
{.5em}
{}

\theoremstyle{myplain}
\newtheorem{theorem}{Theorem}[section]
\newtheorem*{theorem*}{Theorem}
\newtheorem{lemma}[theorem]{Lemma}
\newtheorem{prop}[theorem]{Proposition}
\newtheorem{corollary}[theorem]{Corollary}

\newtheoremstyle{mydefinition} {6pt plus 6pt minus 2pt}
{6pt plus 6pt minus 2pt}
{\itshape}
{}
{\bfseries}
{.}
{.5em}
{}

\theoremstyle{mydefinition}
\newtheorem{definition}[theorem]{Definition}

\newtheoremstyle{myexample} {6pt plus 6pt minus 2pt}
{6pt plus 6pt minus 2pt}
{}
{}
{\scshape}
{.}
{.5em}
{}

\theoremstyle{myexample}

\newtheoremstyle{myremark} {6pt plus 6pt minus 2pt}
{6pt plus 6pt minus 2pt}
{}
{}
{\scshape}
{.}
{.5em}
{}

\theoremstyle{myremark}
\newtheorem{remark}[theorem]{Remark}

\newcommand{\twomatrix}[4]{\ensuremath\begin{pmatrix} #1 & #2 \\ #3 & #4 \end{pmatrix}}

\newcommand{\Z}{\mathbb{Z}}
\newcommand{\Q}{\mathbb{Q}}
\newcommand{\R}{\mathbb{R}}
\newcommand{\C}{\mathbb{C}}

\newcommand{\gl}{\mathfrak{gl}}

\newcommand{\suchthat}{\mid} 
\newcommand{\mapto}{\rightarrow}

\DeclareMathOperator{\Hom}{Hom}

\DeclareMathOperator{\End}{End}

\DeclareMathOperator{\Ann}{Ann}
\newcommand{\id}{\mathrm{id}}
\newcommand{\Id}{\mathrm{Id}}
\newcommand{\abs}[1]{\left|#1\right|}

\renewcommand{\epsilon}{\varepsilon}

\renewcommand{\phi}{\varphi}

\newcommand{\down}{{\mathord\vee}}

\newcommand{\sfP}{{\mathsf{P}}}

\newcommand{\frakh}{{\mathfrak{h}}}

\newcommand{\bolds}{{\boldsymbol{s}}}
\newcommand{\boldt}{{\boldsymbol{t}}}

\newcommand{\boldell}{{\boldsymbol{\ell}}}

\newcommand{\fieldK}{\mathbb{K}}
\newcommand{\unit}{\boldsymbol{1}}
\newcommand{\mult}{\nabla}

\newcommand{\hotimes}{\mathop{\hat{\otimes}}}

\newcommand{\Uqgl}{U_q}
\newcommand{\Uhgl}{U_\hbar}

\renewcommand{\epsilon}{\varepsilon}

\newcommand{\quantumq}{{\boldsymbol{\mathrm{q}}}}

\newcommand{\oDelta}{{\overline{\Delta}}}

\newcommand{\counit}{\boldsymbol{\mathrm{u}}}

\newcommand{\op}{\mathrm{op}}
\newcommand{\ev}{\mathrm{ev}}
\newcommand{\coev}{\mathrm{coev}}

\usepackage{wrapfig} \usepackage{subcaption}

\usepackage[final=true]{hyperref,bookmark}

\hypersetup{
  colorlinks = true,
  urlcolor = true,
  pdfauthor = {Antonio Sartori},
  pdfkeywords = {Lie superalgebras, gl(1|1), Alexander polynomial, Quantum enveloping superalgebras, Quantum invariants},
  pdftitle = {The Alexander polynomial as quantum invariant of links},
  pdfsubject = {},
  pdfpagemode = UseNone,
  bookmarksopen = true,
  bookmarksopenlevel = 3,
  pdfdisplaydoctitle = true }

\title{The Alexander polynomial as quantum invariant of links}
\author{Antonio Sartori}
\address{Mathematisches Institut\\Endenicher Allee 60\\Universit\"at Bonn\\53115 Bonn, Germany}
\email{sartori@math.uni-bonn.de}
\urladdr{http://www.math.uni-bonn.de/people/sartori}
\date{\today}
\keywords{Lie superalgebras, $\gl(1|1)$, Alexander polynomial, Quantum enveloping superalgebras, Quantum invariants.}
\thanks{This work has been supported by the Graduiertenkolleg 1150, funded by the Deutsche Forschungsgemeinschaft.}

\begin{document}

\begin{abstract}
In these notes we collect some results about finite-dimensional representations of $U_q(\gl(1|1))$ and related invariants of framed tangles which are well-known to experts but difficult to find in the literature. In particular, we give an explicit description of the ribbon structure on the category of finite-dimensional $U_q(\gl(1|1))$-representation and we use it to construct the corresponding quantum invariant of framed tangles. We explain in detail why this invariant vanishes on closed links and how one can modify the construction to get a nonzero invariant of framed closed links. Finally we show how to obtain the Alexander polynomial by considering the vector representation of $U_q(\gl(1|1))$.
\end{abstract}
\maketitle

\setcounter{tocdepth}{1}
\microtypesetup{protrusion=false}
\tableofcontents
\microtypesetup{protrusion=true}

\section{Introduction}
\label{sec:introduction}

The Alexander polynomial is a classical invariant of links in $\R^3$, defined first in the 1920s by Alexander
\cite{MR1501429}. Constructed originally in combinatorial terms, it can be defined also in modern language using the homology of a
cyclic covering of the link complement (see for example
\cite{MR1472978}).

The Alexander polynomial can also be defined using the Burau
representation of the braid group (see for example \cite[Chapter
3]{MR2435235}). As well-known to experts, this representation can be
constructed using a solution of the Yang-Baxter equation, which comes
from the action of the $R$-matrix of $U_q(\gl(1|1))$
\cite{MR1133269}  (or alternatively of $U_q(\mathfrak{sl}_2)$ for
$q$ a root of unity; see \cite{MR2255851} for the parallel between
$\gl(1|1)$ and $\mathfrak{sl}_2$).

In other words, the key-point of the construction is the
\emph{braided} structure of the monoidal category of finite-dimensional representations of $U_q(\gl(1|1))$, that is, there is an
action of an $R$-matrix satisfying the braid relation. This can
obviously be used to construct representations of the braid
group. Considering tensor powers of the vector representation of
$U_q(\gl(1|1))$, one obtains in this way the Burau representation of
the braid group. Given a representation of the braid group, one can
extend it to an invariant of links considered as closures of braids by
defining a Markov trace.

In these notes, we exploit this construction a bit further, proving
that the category of finite-dimensional
$U_q(\gl(1|1))$-representations is not only braided, but actually
\emph{ribbon}.  A ribbon category is exactly what one needs to use the
Reshetikhin-Turaev construction \cite{MR1036112} to get invariants of
oriented framed tangles. The advantage of the ribbon structure is that
one can consider arbitrary diagrams of links, and not just braid
diagrams.

To construct a ribbon structure on the category of modules
over some algebra, a possible strategy is to prove that the
algebra is actually a ribbon Hopf algebra. Unfortunately,
similarly to the case of a classical semisimple Lie algebra,
the Hopf algebra $U_q(\gl(1|1))$ is not ribbon. We consider
hence another version of the quantum enveloping algebra,
which we call $U_\hbar(\gl(1|1))$, and which is a
topological algebra over $\C[[\hbar]]$. Roughly speaking,
the relation between $U_q(\gl(1|1))$ and $U_\hbar(\gl(1|1))$
is given by setting $q=e^\hbar$. The price for working with
power series pays off, since $U_\hbar(\gl(1|1))$ is in fact
a ribbon Hopf algebra. By a standard argument, we see that
the $R$-matrix and the ribbon element of $U_\hbar(\gl(1|1))$
act on finite-dimensional representations of $U_q(\gl(1|1))$
and deduce hence the ribbon structure of this category.

Given an oriented framed tangle $T$ and a labeling $\boldell$ of the
strands of $T$ by finite-dimensional irreducible
$U_q(\gl(1|1))$-representations, we get then an invariant
$Q^\boldell(T)$, which is a certain $U_q(\gl(1|1))$-equivariant map. In
particular, restricting to oriented framed links (viewed as special
cases of tangles), we obtain a $\C(q)$-valued invariant.

If we label all the strands by the vector representation of
$U_q(\gl(1|1))$, an easy calculation shows that the corresponding
invariant of oriented framed tangles is actually independent of the
framing and hence is an invariant of oriented tangles (as is
well-known, the same happens for the ordinary $\mathfrak{sl}_n$-invariant).

Unfortunately, when considering invariants of closed links, there is a
little problem we have to take care of. Namely, it follows from the
fact that the category of finite-dimensional $U_q(\gl(1|1))$-modules
is not semisimple (this is true even in the non-quantized case and well-known, see for example \cite{MR2881300} where the blocks of the category of finite-dimensional $\gl(m|n)$-representations are studied in detail) that
the invariant $Q^\boldell(L)$ is zero for all closed links $L$ (see
Proposition \ref{prop:3}). The work-around to avoid this problem is to
choose a strand of the link $L$, cut it and consider the invariant of
the framed $1$-tangle that is obtained in this way (Theorem
\ref{thm:2}). The resulting invariant will be an element of the
endomorphism ring of an irreducible representation (the one that
labels the strand being cut); since this ring can be naturally
identified with $\C(q)$, the invariant that we obtain in this way is
actually a rational function. The construction does not depend on the
strand we cut, but rather on the representation labeling the
strand. In particular for a constant labeling $\boldell$ of all the
components of $L$ we get a true invariant of framed links.

Applying this construction to the constant labeling by the vector representation, one obtains as before an invariant of links. In fact, it is easy to prove that this coincides with the Alexander polynomial (see Theorem \ref{thm:3}).

\medskip The structure of these notes is the following. In Section
\ref{sec:quant-envel-super} we define the quantum enveloping
superalgebras $\Uqgl$ and $\Uhgl$ and explicitly describe the
ribbon structure of the latter. In Section \ref{sec:representations}
we study in detail the category of finite-dimensional representations
of $U_q$. In Section \ref{sec:invariant-links} we construct the
invariants of framed tangles and of links and finally recover the
Alexander polynomial as a consequence of the construction. In the
appendix we collect two technical results about $\Uhgl$.

We want to stress that the content of these short notes is 
well-known to experts. Relations between $U_q(\gl(1|1))$, or more generally $U_q(\gl(n|n))$, and the Alexander polynomial have been noticed, studied and generalized by lots of authors (see for example \cite{MR1034395}, \cite{MR1160372}, \cite{MR1133269}, \cite{MR1371053}, \cite{MR2153122}, \cite{MR2311186}, \cite{MR2640994}). In particular, almost everything we write here
is a special case of what is analyzed in \cite{MR2255851}, where
quantum $\gl(1|1)$- and $\mathfrak{sl}_2$-invariants associated to
arbitrary coloring of tangles are studied in detail. Our aim is to
provide, from a purely representation theoretical point of view, a
short but complete and self-contained explanation of how the Alexander
polynomial arises as quantum invariant corresponding to the vector
representation of $U_q(\gl(1|1)$, including a full proof of the
ribbon structure of $U_q(\gl(1|1))$.

\subsubsection*{Acknowledgements} The present work is part of the author's PhD thesis. The author would like to thank his advisor Catharina Stroppel for her help and support. The author would also like to thank the anonymous referee for many helpful comments.

\section{The quantum enveloping superalgebras \texorpdfstring{$\Uqgl$}{Uq(gl(1|1))} and \texorpdfstring{$\Uhgl$}{Uh(gl(1|1))}}
\label{sec:quant-envel-super}

We recall the definition of the quantum enveloping algebra $\Uqgl=U_q(\gl(1|1))$ and of its Hopf superalgebra structure. We define then the $\hbar$-version $U_\hbar=U_\hbar(\gl(1|1))$ and prove that it is a ribbon Hopf superalgebra.

In the following, as usual, by a \emph{super}object (for example vector space, algebra, Lie algebra, module) we will mean a $\Z/2\Z$-graded object. If $X$ is such a superobject we will use the notation $\abs{x}$ to indicate the degree of a homogeneous element $x \in X$. Elements of degree $0$ are called \emph{even}, while elements of degree $1$ are called \emph{odd}.
 We stress that whenever we write $\abs{x}$ we will always be assuming $x$ to be homogeneous.

Throughout the section we will use some standard facts about Hopf superalgebras. The analogous statements in the non-super setting can be found for example in \cite{MR1300632}, \cite{MR1321145}, \cite{MR1881401}. The proofs carry directly over to the super case.

\subsection{The Lie superalgebra \boldmath\texorpdfstring{$\gl(1|1)$}{gl(1|1)}}
\label{sec:lie-super-gl11}

Let $\C^{1|1}$ be the 2-dimensional complex vector space on basis
$u_0$, $u_1$ viewed as a super vector space by setting $\abs{u_0}=0$
and $\abs{u_1}=1$.  The space of linear endomorphisms of $\C^{1|1}$ inherits a $\Z/2\Z$-grading and turns into a Lie superalgebra $\gl(1|1)$ with the supercommutator
\begin{equation}
  \label{eq:1}
  [a,b] = ab - (-1)^{\abs{a}\abs{b}} ba.
\end{equation}

As a Lie superalgebra, $\gl(1|1)$ is four-dimensional and generated by the elements
\begin{equation}
  \label{eq:2}
  h_1 = \twomatrix{1}{0}{0}{0}, \quad h_2= \twomatrix{0}{0}{0}{1}, \quad e = \twomatrix{0}{1}{0}{0}, \quad f= \twomatrix{0}{0}{1}{0}
\end{equation}
whit $\abs{h_1}=\abs{h_2}=0$ and $\abs{e}=\abs{f}=1$, 
subject to the defining relations
\begin{equation}
  \label{eq:3}
  \begin{aligned}
    [h_1,e]&=e,& [h_2,e]&=-e, &[h_2,f]&=f, &[h_1,f]&=-f,\\
    [h_1,h_2]& = 0,&  [e,f]&=h_1+h_2,& [e,e]&=0,& [f,f]&=0.
  \end{aligned}
\end{equation}

Let $\frakh \subset \gl(1|1)$ be the Cartan subalgebra consisting of
all diagonal matrices.  In $\frakh^*$ let $\epsilon_1,\epsilon_2$ be
the basis dual to $h_1,h_2$. On $\frakh^*$ we define a non-degenerate
symmetric bilinear form by setting on the basis
\begin{equation}
  \label{eq:47}
  (\epsilon_i,\epsilon_j) =
  \begin{cases}
    1 & \text{if } i=j=1,\\
    -1 & \text{if } i=j=2,\\
    0 & \text{if } i \neq j.
  \end{cases}
\end{equation}
The \emph{roots} of $\gl(1|1)$ are $\alpha=\epsilon_1 - \epsilon_2$
and $-\alpha$; we choose $\alpha$ to be the positive \emph{simple
  root}. Denote by $\sfP=\Z\epsilon_1 \oplus \Z\epsilon_2 \subset
\frakh^*$ the \emph{weight lattice} and by $\sfP^*=\Z h_1 \oplus \Z
h_2 \subset \frakh$ its dual.

\subsection{The quantum enveloping superalgebra}
\label{sec:quant-envel-super-1}

The \emph{quantum enveloping superalgebra} $U_q=U_q(\gl(1|1))$ is defined to be
the unital superalgebra over $\C (q)$ with generators $E$, $F$, $\quantumq^h\,
(h \in \sfP^*)$ in degrees $\abs{\quantumq^h}=0$, $\abs{E}=\abs{F}=1$ subject to the relations
\begin{equation}
\begin{gathered}
  \quantumq^0=1, \qquad \quantumq^h \quantumq^{h'} = \quantumq^{h+h'} \qquad \text{for }h,h' \in \sfP^*,\\
  \quantumq^h E=q^{\langle h, \alpha \rangle} E \quantumq^h, \qquad \quantumq^h F=q^{-\langle h,\alpha \rangle}F\quantumq^h \qquad \text{for }h \in \sfP^*,\\
  EF+FE= \frac{K-K^{-1}}{q-q^{-1}} \qquad \text{where }K=\quantumq^{h_1+h_2},\\
  E^2=F^2=0.
\end{gathered}\label{reps:eq:1}
\end{equation}
The elements $\quantumq^h$, which for the moment are formal symbols, can be interpreted in terms of exponentials in $U_\hbar(\gl(1|1))$ (see following). Notice that all elements $\quantumq^h$ for $h \in \sfP^*$ are linear combination of $\quantumq^{h_1}$ and $\quantumq^{h_2}$, so that $\Uqgl$ is finitely generated. Note also that $K$ is a central element of $\Uqgl$, very much in contrast to $U_q(\mathfrak{sl}_2)$.

\subsection{Hopf superalgebras}
\label{sec:hopf-superalgebras}

We recall that if $A$ is a superalgebra then $A \otimes A$
can be given a superalgebra structure by declaring $(a \otimes b)(c
\otimes d) = (-1)^{\abs{b}\abs{c}} ac \otimes bd$. If $M$
and $N$ are $A$-supermodules, then $M \otimes N$ becomes an $(A \otimes
A)$-supermodule with action $(a \otimes b)\cdot (m \otimes n)=
(-1)^{\abs{b} \abs{m}} am \otimes bn$ for $a,b \in A$, $m\in M$, $n
\in N$.

A \emph{super bialgebra} $B$ over a field $\fieldK$ is then a unital superalgebra which is also a coalgebra, such that the counit $\counit: B \mapto \fieldK$ and the comultiplication $\Delta: B \mapto B \otimes B$ are homomorphism of superalgebras (and are homogeneous of degree $0$).
A \emph{Hopf superalgebra} $H$ is a super bialgebra equipped with a $\fieldK$-linear \emph{antipode} $S: H \mapto H$ (homogeneous of degree $0$) such that the usual diagram
\begin{equation}
  \label{eq:50}
  \begin{tikzpicture}[baseline=(current bounding box.center)]
  \matrix (m) [matrix of math nodes, row sep=2em, column
  sep=5.5em, text height=1.5ex, text depth=0.25ex] {
    H \otimes H & H & H \otimes H \\
    & \fieldK & \\
    H \otimes H & H & H \otimes H\\};
  \path[->] (m-1-2) edge node[above] {$ \Delta $} (m-1-1);
  \path[->] (m-1-2) edge node[auto] {$ \Delta $} (m-1-3);
  \path[->] (m-3-1) edge node[auto] {$ \mult $} (m-3-2);
  \path[->] (m-3-3) edge node[above] {$ \mult $} (m-3-2);
  \path[->] (m-1-1) edge node[left] {$ S \otimes \id $} (m-3-1);
  \path[->] (m-1-3) edge node[auto] {$ \id \otimes S $} (m-3-3);
  \path[->] (m-1-2) edge node[auto] {$ \counit $} (m-2-2);
  \path[->] (m-2-2) edge node[auto] {$ \unit $} (m-3-2);
\end{tikzpicture}
\end{equation}
commutes, where $\mult: H \otimes H \mapto H$ and $\unit: \fieldK \mapto H$ are the multiplication and unit of the algebra structure.

If $H$ is a Hopf superalgebra and $M$, $N$ are (finite-dimensional) $H$-supermodules then the comultiplication
$\Delta$ makes it possible to give $M \otimes N$ an
$H$-module structure by letting
\begin{equation}
x \cdot (m \otimes n) = \Delta(x) (m \otimes n) = \sum_{(x)} (-1)^{\abs{x_{(2)}} \abs{\vphantom{x_{(2)}}m}} x_{(1)} m \otimes x_{(2)} n \label{eq:51}
\end{equation}
for $x \in H$, $m \otimes n \in M \otimes N$, where we use Sweedler notation $\Delta(x) = \sum_{(x)} x_{(1)} \otimes x_{(2)}$. Notice in particular that  signs appear.
The antipode $S$, moreover, allows to turn $M^* = \Hom_\fieldK(M, \fieldK)$
into an $H$-module via
\begin{equation}
(x \phi) (v) = (-1)^{\abs{\phi} \abs{x}} \phi (S(x)v)\label{eq:52}
\end{equation}
 for $x
\in H$, $\phi \in M^*$. Again, notice that a sign appears. A good rule to keep in mind is that
a sign appears whenever an odd element steps over some other
odd element. A good reference for sign issues is
\cite[Chapter 3]{MR1632008}.

\subsection{The Hopf superalgebra structure on \boldmath$U_q$}
Let us now go back to $\Uqgl$. We define a \emph{comultiplication} $\Delta: U_q \mapto U_q \otimes U_q $, a \emph{counit} $\counit: U_q \mapto \C(q)$ and an \emph{antipode} $S: U_q \mapto U_q$ by setting on the generators
\begin{equation}
\begin{aligned}
  \Delta(E)&= E \otimes K^{-1}+1 \otimes E, & \Delta(F)&=F \otimes 1 + K \otimes F,\\
  S(E)&=-EK, & S(F)&=- K^{-1}F,\\
  \Delta(\quantumq^h)&=\quantumq^h \otimes \quantumq^h, &   S(\quantumq^h)&=\quantumq^{-h},\\
  \counit(E)& =\counit(F)=0, & \counit(\quantumq^h)&=1,
\end{aligned}\label{reps:eq:2}
\end{equation}
and extending $\Delta$ and $\counit$ to algebra homomorphisms and $S$ to an algebra anti-homomorphism. We have then:
\begin{prop}
  \label{prop:1}
  The maps $\Delta$, $\counit$ and $S$ turn $U_q$ into a Hopf superalgebra.
\end{prop}

\begin{proof}
  This is a straightforward calculation.
\end{proof}

Notice that from the centrality of $K$ it follows that $S^2=\id$; this
is a special property of $\Uqgl$, that for instance does not hold in $U_q(\mathfrak{gl}(m|n))$ for general $m,n$ (see
\cite{MR1694051} for a definition of the general linear quantum supergroup).

We define a \emph{bar involution} on $\Uqgl$ by setting:
\begin{equation}
  \label{reps:eq:3}
  \overline E=E, \qquad \overline F=F, \qquad \overline { \quantumq^h}  = \quantumq^{-h},
  \qquad \overline q=q^{-1}.
\end{equation}
Note that $\overline{\Delta}=(\overline{\phantom{x}} \otimes \overline{\phantom{x}}) \circ \Delta
\circ \overline{\phantom{x}}$ defines another comultiplication on $\Uqgl$, and by
definition $\overline \Delta(\overline x) = \overline{ \Delta(x)}$ for all $x \in \Uqgl$.

\subsection{The Hopf superalgebra \boldmath$\Uhgl$}
\label{sec:completion}

Our goal is to construct a ribbon category of representations of $\Uqgl$, so that we can define link invariants. The main ingredient is the $R$-matrix.
Unfortunately, as usual, it is not possible to
construct a universal $R$-matrix for $\Uqgl$; instead, we need to
consider the $\hbar$-version of the quantum enveloping superalgebra, which
we will denote by $\Uhgl$ and which is a $\C[[\hbar]]$-superalgebra
completed with respect to the $\hbar$-adical topology. We will prove that $\Uhgl$ is a ribbon algebra. Then, using a standard argument of Tanisaki \cite{MR1187582}, we obtain a ribbon structure on the category of finite-dimensional $\Uqgl$-representations.
For details
about topological $\C[[\hbar]]$-algebras we refer to \cite[Chapter
XVI]{MR1321145}. We will denote by the symbol $\hotimes$ the \emph{completed} tensor product of topological $\C[[\hbar]]$-algebras.

We define $\Uhgl$ to be the unital $\C[[\hbar]]$-algebra topologically generated by the elements $E, F, H_1, H_2$ in degrees $\abs{H_1}=\abs{H_2}=0$, $\abs{E}=\abs{F}=1$ subject to the relations
\begin{equation}\label{eq:4}
\begin{gathered}
   H_1H_2=H_2H_1,\\
 H_i E - E H_i = \langle H_i, \alpha \rangle E, \qquad    H_i F - F H_i = - \langle H_i, \alpha \rangle F,\\
  EF+FE= \frac{e^{\hbar (H_1+H_2)}-e^{-\hbar (H_1+H_2)}}{e^{\hbar} - e^{-\hbar}}, \qquad  E^2=F^2=0.
\end{gathered}
\end{equation}
Note that although $e^\hbar- e^{-\hbar}$ is not invertible, it is the product of $\hbar$ and an invertible element of $\C[[\hbar]]$, hence the fourth relation makes sense.

Although the relation between $\Uqgl$ and $\Uhgl$ is technically not easy to formalize (see \cite{MR1300632} for details), one should keep in mind the following picture:
\begin{equation}
  \label{eq:36}
  \begin{aligned}
    q & \longleftrightarrow e^{\hbar},\\
    \quantumq^{h_i} & \longleftrightarrow e^{\hbar H_i}.
  \end{aligned}
\end{equation}
This also explains why we use the symbols $\quantumq^h$ as generators for $\Uqgl$. In the following, we set $q=e^\hbar$ as an element of $\C[[\hbar]]$ and $K=e^{\hbar (H_1+H_2)}$ as an element of $\Uhgl$. 

As before, we define a \emph{comultiplication} $\Delta: U_\hbar \mapto U_\hbar \hotimes U_\hbar$, a \emph{counit} $\counit: U_\hbar \mapto \C[[\hbar]]$ and an \emph{antipode} $S: U_\hbar \mapto U_\hbar$ by setting for the generators
\begin{equation}
\begin{aligned}
  \Delta(E)&= E \otimes K^{-1} +1 \otimes E, & \Delta(F)&=F \otimes 1 + K \otimes F,\\
  S(E)&=-EK, & S(F)&=- K^{-1}F,\\
  \Delta(H_i)&=H_i \otimes 1+1\otimes H_i, &   S(H_i)&=-H_i,\\
  \counit(E)& =\counit(F)=0, & \counit(H_i)&=0,
\end{aligned}\label{eq:13}
\end{equation}
and extending $\Delta$ and $\counit$ to algebra homomorphisms and $S$ to an algebra anti-homomorphism. We have then:
\begin{prop}
  \label{prop:4}
  The maps $\Delta$, $\counit$ and $S$ turn $U_\hbar$ into a Hopf superalgebra.
\end{prop}
The proof requires precisely the same calculations as the proof of Proposition~\ref{prop:1}.

As for $\Uqgl$, we define a \emph{bar involution} on $\Uhgl$ by setting:
\begin{equation}
  \label{eq:48}
  \overline E=E, \qquad \overline F=F, \qquad \overline {H_i} = H_i,
  \qquad \overline \hbar=-\hbar.
\end{equation}
As before, $\overline{\Delta}=(\overline{\phantom{x}} \otimes \overline{\phantom{x}}) \circ \Delta
\circ \overline{\phantom{x}}$ defines another comultiplication on $\Uhgl$, and by
definition $\overline \Delta(\overline x) = \overline{ \Delta(x)}$ for all $x \in \Uhgl$.

\subsection{The braided structure}
\label{reps:sec:quas-struct}
We are going to recall the \emph{braided Hopf superalgebra
  structure} (cf.\ \cite{MR1892766}, \cite{MR1881401}) of $\Uhgl$. The main ingredient is the universal $R$-matrix, which has
been explicitly computed by Khoroshkin and Tolstoy
(cf.\ \cite{MR1134942}). We adapt their definition to our notation.\footnote{Our comultiplication is the opposite of \cite{MR1134942},
  hence we have to take the opposite $R$-matrix, cf.\ also
  \cite[Chapter 8]{MR1321145}.}

We define $R=\Theta \Upsilon \in \Uhgl \hotimes \Uhgl$ where
\begin{align}
  \Upsilon &
= e^{\hbar (H_1 \otimes H_1 - H_2 \otimes H_2)},
  \label{reps:eq:33} \\
  \Theta & = 1 + (q - q^{-1} ) F \otimes E.\label{reps:eq:34}
\end{align}
Notice that the expression for $\Upsilon$ makes sense as an element of the completed tensor product $\Uhgl \hotimes \Uhgl$.
Recall that a vector $w$ in some representation $W$ of $\Uhgl$ is said to be a \emph{weight vector} of \emph{weight} $\mu$ if $H_i w = \langle H_i, \mu \rangle w$ for $i=1,2$. 
The element $\Upsilon$ is then characterized by the property that it
acts on a weight vector $w_1 \otimes w_2$ by  $q^{(\mu_1,\mu_2)} = e^{\hbar (\mu_1,\mu_2)}$,
if $w_1$ and $w_2$ have weights $\mu_1$ and $\mu_2$ respectively.

The element $\Theta$ is called the \emph{quasi $R$-matrix}; it is easy to check that it satisfies
\begin{equation}
  \Theta \overline \Theta = \overline \Theta \Theta = 1 \otimes 1.\label{reps:eq:8}
\end{equation}
It follows in particular that $R$ is invertible with inverse $R^{-1} = \Upsilon^{-1} \Theta^{-1}= \Upsilon^{-1} \overline \Theta$.

Recall that a bialgebra $B$ is called \emph{quasi-cocommutative} (\cite[Definition VIII.2.1]{MR1321145}) if there exists an invertible element $R \in B \otimes B$
such that for all $x \in B$ we have $\Delta^\op (x) = R \Delta(x) R^{-1}$, where  $\Delta^\op$ is the opposite comultiplication $\Delta^\op=\sigma \circ
\Delta$ with $\sigma (a \otimes b) = (-1)^{\abs{a} \abs{b}} (b \otimes
a)$.

\begin{lemma}
  \label{lem:3}
For all $x \in \Uhgl$ we have
\begin{equation}
  \label{reps:eq:5}
  R \Delta (x) = \Delta^\op(x) R.
\end{equation}
Hence the Hopf algebra $\Uhgl$ is quasi-cocommutative.
\end{lemma}

\begin{proof}
  Using Lemma \ref{lem:1} we compute
  \begin{equation*}
    R \Delta (x) = \Theta \Upsilon \Delta(x) = \Theta \overline \Delta^\op(x) \Upsilon = \Delta^\op(x) \Theta \Upsilon= \Delta^\op(x) R. \qedhere
  \end{equation*}
\end{proof}

A quasi-cocommutative Hopf algebra is called \emph{braided} or \emph{quasi-triangular} if the following quasi-triangularity identities hold:
\begin{equation}
  \label{reps:eq:6}
    (\Delta \otimes \id) (R) = R_{13} R_{23} \qquad \text{and}
    \qquad (\id \otimes  \Delta)( R) = R_{13}R_{12}.\qedhere
\end{equation}
In this case, the element $R$ is called \emph{universal R-matrix}.

\begin{prop}
  \label{prop:2}
  The Hopf superalgebra $\Uhgl$ is braided.
\end{prop}

\begin{proof}
  Since
  \begin{equation*}
    \label{eq:43}
    (\Delta \otimes \id)(\Upsilon) = e^{\hbar (H_1 \otimes 1 \otimes H_1 + 1 \otimes H_1 \otimes H_1 - H_2 \otimes 1 \otimes H_2 - 1 \otimes H_2 \otimes H_2)} = \Upsilon_{13}\Upsilon_{23}
  \end{equation*}
  we can compute using Lemma \ref{lem:5}
  \begin{equation*}
    \label{eq:42}
    (\Delta \otimes \id)(R) = (\Delta \otimes \id)(\Theta)\cdot(\Delta \otimes \id)( \Upsilon) = \Theta_{13} \Upsilon_{13} \Theta_{23} \Upsilon_{13}^{-1}  \Upsilon_{13} \Upsilon_{23} = R_{13} R_{23}.
  \end{equation*}
  Similarly we get $(\id \otimes \Delta)(R)= R_{13}R_{12}$.
\end{proof}

As an easy consequence of the braided structure, the following Yang-Baxter equation holds (see \cite[Theorem VIII.2.4]{MR1321145} or \cite[Proposition 4.2.7]{MR1300632}):
\begin{equation}
  \label{reps:eq:7}
  R_{12}R_{13}R_{23}=R_{23}R_{13}R_{12}.
\end{equation}

\subsection{The ribbon structure}
\label{reps:sec:ribbon-structure}

Write $R=\sum_r a_r \otimes b_r$ and define
\begin{equation}
  \label{reps:eq:64}
  u = \sum_r (-1)^{\abs{a_r}\abs{b_r}}S(b_r) a_r \in \Uhgl.
\end{equation}
Then (cf.\ \cite[Proposition 4.2.3]{MR1300632}) $u$ is invertible and we have
\begin{equation}
S^2(x)=uxu^{-1}  \qquad \text{for all }x \in \Uhgl.\label{reps:eq:65}
\end{equation}
In our case, in particular, since $S^2 = \id$, the element $u$ is central. By an easy explicit computation, we have
\begin{equation}
  \label{reps:eq:66}
  u = (1 + (q-q^{-1})EKF) e^{\hbar (H_2^2-H_1^2)}
\end{equation}
and
\begin{equation}
  \label{reps:eq:68}
  S(u)=e^{\hbar (H_2^2-H_1^2)}(1-(q-q^{-1})FK^{-1}E).
\end{equation}

We recall that a braided Hopf superalgebra $A$ is called \emph{ribbon}
(cf.\ \cite[Chapter 4]{MR1881401} or
\cite[\textsection{}4.2.C]{MR1300632}) if there is an even central
element $v \in A$ such that
\begin{equation}
  \begin{gathered}
    v^2=uS(u),\qquad  \counit(v)=1, \qquad S(v)=v,\\
    \Delta(v)=(R_{21}R_{12})^{-1}(v\otimes v).
  \end{gathered} \label{eq:44}
\end{equation}
In
$\Uhgl$ let
\begin{equation}
  \label{reps:eq:67}
  v=K^{-1}u=uK^{-1} = (K^{-1} + (q-q^{-1})EF) e^{\hbar (H_2^2-H_1^2)}.
\end{equation}
Then we have:
\begin{prop}
  \label{reps:prop:2}
  With $v$ as above, $\Uhgl$ is a ribbon Hopf superalgebra.
\end{prop}
\begin{proof}
  Since both $u$ and $K^{-1}$ are central, so is $v$.  Let us check that  $S(u)=uK^{-2}$.
  Indeed we have
  \begin{equation}
    \label{eq:19}
    \begin{aligned}
      u &= (1 + (q-q^{-1})EFK) e^{\hbar(H_2^2-H_1^2)} \\
      &=  e^{\hbar(H_2^2-H_1^2)} (1 + (q-q^{-1})EFK) \\
      &= e^{\hbar(H_2^2-H_1^2)}  (1 + (K-K^{-1})K - (q-q^{-1}) FEK)\\
      &= e^{\hbar(H_2^2-H_1^2)}  (K^2 - (q-q^{-1}) FEK) = S(u)K^2.   
    \end{aligned}
  \end{equation}

  It follows then immediately that $v^2 = u^2K^{-2} = u S(u)$ and
  $S(v)=S(u)K = u K^{-1} = v$.

  The relations $\Delta(v)=(R_{21}R_{12})^{-1}(v\otimes v)$ and
  $\counit(v)=1$ follow from
  analogous relations for $u$, that hold for every quasi-triangular
  Hopf superalgebra (see \cite[Proposition 4.3]{MR1881401}).
\end{proof}

\section{Representations}
\label{sec:representations}
We define a parity function $\abs{\cdot}:\sfP \mapto \Z/2\Z$ on the weight lattice by setting $\abs{\epsilon_1}=0$, $\abs{\epsilon_2}=1$ and extending additively. By a \emph{representation} of $\Uqgl$ we mean from now on a finite-dimensional
$\Uqgl$-supermodule with a decomposition into weight spaces $M =
\bigoplus_{\lambda \in \sfP} M_\lambda$ with integral weights $\lambda \in \sfP$, such that $\quantumq^h$ acts as
$q^{\langle h,\lambda \rangle}$ on $M_\lambda$. We suppose further that $M$ is $\Z/2\Z$-graded, and the grading is uniquely determined by the requirement that $M_\lambda$ is in degree $\abs{\lambda}$.

\subsection{Irreducible representations}
\label{sec:irred-repr}

It is not difficult to find all simple representations: up to
isomorphism they are indexed by their highest weight $\lambda \in
\sfP$.
If $\lambda \in \Ann (h_1+h_2)$, then 
the simple representation with highest weight $\lambda$ is one-dimensional, generated by a vector
$v^\lambda$ in degree $\abs{v^\lambda}=\abs{\lambda}$
with
\begin{align}\label{reps:eq:12}
  Ev^\lambda&=0, &  Fv^\lambda&=0, & \quantumq^h v^\lambda &=q^{\langle h, \lambda \rangle} v^\lambda, &  Kv^\lambda & = v^\lambda.
\end{align}
We will denote this representation by $\C(q)_\lambda$, to emphasize
that it is just a copy of $\C(q)$ on which the action is twisted
by the weight $\lambda$. In particular for $\lambda = 0$ we
have the trivial representation $\C(q)_0$, that we will simply denote
by $\C(q)$ in the following.

If $\lambda \notin \Ann (h_1 + h_2)$ then the simple representation
$L(\lambda)$ with highest weight $\lambda$ is two-dimensional; we
denote by $v^\lambda_0$ its highest weight vector. Let us also
introduce the following notation that will be useful later:
\begin{equation}
  \label{reps:eq:43}
  q^\lambda = q^{\langle h_1+h_2,\lambda\rangle}, \qquad [\lambda] =
  [\langle h_1+h_2, \lambda \rangle],
\end{equation}
where, as usual, $[k]$ is the quantum number defined by
\begin{equation}\label{eq:173} [k]=\frac{q^k-q^{-k}}{q-q^{-1}} = q^{-k+1} + q^{-k +
    3} + \cdots + q^{k-3} + q^{k-1}.
\end{equation}
Even if the second equality holds only for $k>0$, we define $[k]$ for all integers $k$ using the first equality; in particular
we have $[-k]=-[k]$.

 Then
$L(\lambda) = \C(q)\langle v^\lambda_0 \rangle \oplus \C(q) \langle
v^\lambda_1 \rangle$ with $\abs{v^\lambda_0}=\abs{\lambda}$, $\abs{v^\lambda_1}=\abs{\lambda}+1$ and
\begin{equation}
\begin{aligned}
  Ev^\lambda_0&=0, &  Fv^\lambda_0&=[\lambda] v^\lambda_1, & \quantumq^h v^\lambda_0
  &=q^{\langle h,\lambda\rangle} v^\lambda_0, &  Kv^\lambda_0 & = q^{ \lambda } v^\lambda_0,\\
  Ev^\lambda_1&= v^\lambda_0, & Fv^\lambda_1&=0,      & \quantumq^h v^\lambda_1 &=q^{\langle
    h, \lambda - \alpha \rangle} v^\lambda_1, &  Kv^\lambda_1 & = q^{\lambda} v^\lambda_1.
\end{aligned}\label{reps:eq:35}
\end{equation}

\begin{remark}
  \label{rem:1}
  As a remarkable property of $\Uqgl$, we notice that since
  $E^2=F^2=0$ \emph{all} simple $\Uqgl$-modules (even the ones with
  non-integral weights) are finite-dimensional. In fact, formulas
  \eqref{reps:eq:35} define two-dimensional simple $\Uqgl$-modules for
  all complex weights $\lambda \in \C \epsilon_1 \oplus \C \epsilon_2$
  such that $\langle h_1+h_2, \lambda \rangle \neq 0$.
\end{remark}

In the following, we set
\begin{equation}
\sfP'=\{\lambda \in \sfP \suchthat \lambda \notin \Ann (h_1+h_2)\}\label{eq:179}
\end{equation}
and we will mostly consider two-dimensional simple representations $L(\lambda)$ for $\lambda \in \sfP'$.
Also, $\sfP^\pm = \{\lambda \in \sfP \suchthat \langle h_1+h_2,\lambda \rangle \gtrless0\}$ will be the set of positive/negative weights and $\sfP' = \sfP^+ \sqcup \sfP^-$.

\begin{remark}
  Note that in analogy with the classical Lie situation, we can set
  $\alpha^{\down} = h_1+h_2$. Then $e,f,\alpha^\down$ generate the Lie superalgebra $\mathfrak{sl}(1|1)$ inside $\gl(1|1)$. We work with $\gl(1|1)$ and not with $\mathfrak{sl}(1|1)$ since the latter is not reductive, but nilpotent.\label{rem:3}
\end{remark}

\subsection{Decomposition of tensor products}
\label{sec:decomp-tens-prod}

The following lemma is the first step to decompose
a tensor product of $\Uqgl$-representations:

\begin{lemma}\label{reps:lem:1}
  Let $\lambda, \mu \in \sfP'$ and suppose also $\lambda+\mu \in \sfP'$. Then we have
  \begin{equation}
    L(\lambda) \otimes L(\mu) \cong L(\lambda + \mu) \oplus
    L(\lambda+\mu-\alpha).
    \label{reps:eq:36}
  \end{equation}
\end{lemma}
\begin{proof}
   Under our assumptions, the vectors
  \begin{align}
    E(v^\lambda_1 \otimes v^\mu_1) & = v^\lambda_0 \otimes q^{- \mu } v^\mu_1 + (-1)^{\abs{\lambda}+1} v^\lambda_1 \otimes
    v^\mu_0, \label{eq:174}\\
    F(v^\lambda_0 \otimes v^\mu_0) & = [\lambda]
    v^\lambda_1 \otimes v^\mu_0 + (-1)^{\abs{\lambda}} q^{\lambda} v^\lambda_0 \otimes [\mu]v^\mu_1\label{eq:175}
  \end{align}
  are linearly
  independent. One can verify easily that $v^\lambda_1 \otimes v_1^\mu$
  and $E(v^\lambda_1 \otimes v_1^\mu)$ span a module isomorphic to
  $L(\lambda+\mu-\alpha)$, while $v^\lambda_0 \otimes v^\mu_0$ and
  $F(v^\lambda_0 \otimes v^\mu_0)$ span a module isomorphic to
  $L(\lambda+\mu)$.
\end{proof}

On the other hand, we have:
\begin{lemma}
  \label{lem:2}
  Let $\lambda, \mu \in \sfP'$ and suppose $\lambda + \mu \in \Ann(h_1+h_2)$. Then the representation $M=L(\lambda) \otimes L(\mu)$ is indecomposable and has a filtration
  \begin{equation}
    \label{eq:14}
    {0}=M_0 \subset M_1 \subset M_2 \subset M
  \end{equation}
  with successive quotients
  \begin{equation}\label{eq:15}
M_1 \cong \C(q)_{\nu}, \quad M_2/M_1 \cong \C(q)_{\nu-\alpha} \oplus \C(q)_{\nu+\alpha}, \quad M/M_2 \cong \C(q)_{\nu}
\end{equation}
where $\nu=\lambda+\mu-\alpha$.

Moreover, $L(\lambda') \otimes L(\mu') \cong L(\lambda) \otimes
L(\mu)$ for any $\lambda', \mu' \in \sfP'$ 
such that $\lambda'+\mu'=\lambda+\mu$.
\end{lemma}

\begin{proof}
  Since $\lambda + \mu \in \Ann(h_1+h_2)$ we have $q^\lambda=q^{-\mu}$ and $[\lambda]=-[\mu]$.
  Using \eqref{eq:174} and \eqref{eq:175} we get that
  \begin{equation}
    \label{eq:16}
    F(v_0^\lambda \otimes v_0^\mu) = (-1)^{\abs{\lambda}+1} [\lambda] E(v_1^\lambda \otimes v_1^\mu).
  \end{equation}
  In particular, since $E^2=F^2=0$, the vector $F(v_0^\lambda \otimes v_0^\mu)$ generates a one-dimensional submodule $M_1\cong\C(q)_{\lambda+\mu-\alpha}$ of $M$. It follows then that the images of $v_0^\lambda \otimes v_0^\mu$ and $v_1^\lambda \otimes v_1^\mu$ in $M/M_1$ generate two one-dimensional submodules isomorphic to $\C(q)_{\lambda+\mu}$ and $\C(q)_{\lambda+\mu-2\alpha}$ respectively. Let therefore $M_2$ be the submodule of $M$ generated by $v_0^\lambda \otimes v_0^\mu$ and $v_1^\lambda \otimes v_1^\mu$. Then $M/M_2$ is a one-dimensional representation isomorphic to $\C(q)_\nu$.

The last assertion follows easily since both $L(\lambda) \otimes L(\mu)$ and $L(\lambda') \otimes L(\mu')$ are isomorphic as left $\Uqgl$-modules to $\Uqgl/I$ where $I$ is the left ideal generated by the elements $\quantumq^h-q^{\langle h, \nu\rangle}$ for $h \in \sfP$.
\end{proof}

\subsection{The dual of a representation}
\label{sec:dual-representation}

Let us consider now the dual $L(\lambda)^*$ of the
representation $L(\lambda)$ and let
$(v_0^\lambda)^*,(v_1^\lambda)^*$ be the basis dual to the
standard basis $v_0^\lambda,v_1^\lambda$, with $\abs{(v_0^\lambda)^*}=\abs{v_0}=\abs{\lambda}$ and $\abs{(v_1^\lambda)^*}=\abs{v_1}=\abs{\lambda}+1$ By explicit
computation, the action of $\Uqgl$ on $L(\lambda)^*$ is
given by:
\begin{equation}
\begin{aligned}
  E(v^\lambda_0)^*&=-(-1)^{\abs{\lambda}} q^\lambda (v^\lambda_1)^*, &  E(v^\lambda_1)^*&=0, \\
     F(v^\lambda_0)^*&=0, & F(v^\lambda_1)^*&= (-1)^{\abs{\lambda}} [\lambda]q^{-\lambda}(v^\lambda_0)^*, \\  
 \quantumq^h (v^\lambda_0)^* &=q^{-\langle
    h,\lambda \rangle} (v^\lambda_0)^*,
 & \quantumq^h (v^\lambda_1)^* &=q^{-\langle h ,\lambda-\alpha\rangle} (v^\lambda_1)^* .
\end{aligned}\label{reps:eq:32}
\end{equation}
The assignment
\begin{equation}
  \label{eq:17}
  \begin{aligned}
    L(\alpha-\lambda) & \longrightarrow L(\lambda)^* \\
    v^{\alpha-\lambda}_0 & \longmapsto -(-1)^{\abs{\lambda}} q^{\lambda} (v^\lambda_1)^*\\
    v^{\alpha-\lambda}_1 & \longmapsto (v^\lambda_0)^*
  \end{aligned}
\end{equation}
defines a $\Q(q)$-linear map which is in fact an isomorphism of $\Uqgl$-modules
\begin{equation}
  \label{reps:eq:17}
  L(\lambda)^* \cong L(\alpha-\lambda).
\end{equation}

\begin{remark}\label{rem:2}
  Together with Lemma \ref{lem:2} it follows that $L(\lambda) \otimes
  L(\lambda)^*$ is an indecomposable representation. In the filtration \eqref{eq:14}, the submodule $M_1$ is the image of the coevaluation map $\C(q) \mapto L(\lambda) \otimes L(\lambda)^*$ while the submodule $M_2$ is the kernel of the evaluation map $L(\lambda) \otimes L(\lambda)^* \mapto \C(q)$, see \eqref{eq:8} and \eqref{eq:9} below.
\end{remark}

\begin{remark}
  \label{reps:rem:3}
  At this point it is probably useful to recall that the natural
  isomorphism $V \cong V^{**}$ for a super vector space is given by $x
  \longmapsto ( \phi \mapsto (-1)^{\abs{x}\abs{\phi}} \phi(x))$.
\end{remark}

\subsection{The vector representation}
\label{reps:sec:regul-repr}

The vector representation of $\Uqgl$ is isomorphic to $L(\epsilon_1)$ with its standard basis
$v^{\epsilon_1}_0,v^{\epsilon_1}_1$ and grading given by $\abs{v^{\epsilon_1}_0}=0, \abs{v^{\epsilon_1}_1}=1$, and the
action of $\Uqgl$ is given by
\begin{equation}
\begin{aligned}
  Ev^{\epsilon_1}_0&=0, &  Fv^{\epsilon_1}_0&=v^{\epsilon_1}_1, & \quantumq^h v^{\epsilon_1}_0 &=q^{\langle h,\epsilon_1 \rangle} v^{\epsilon_1}_0, & Kv^{\epsilon_1}_0&=qv^{\epsilon_1}_0, \\
  Ev^{\epsilon_1}_1&=v^{\epsilon_1}_0, &  Fv^{\epsilon_1}_1&=0, & \quantumq^h v^{\epsilon_1}_1 &=q^{\langle h,\epsilon_2 \rangle} v^{\epsilon_1}_1, & Kv^{\epsilon_1}_1&=qv^{\epsilon_1}_1.
\end{aligned}\label{reps:eq:14}
\end{equation}

For $L(\epsilon_1)^{\otimes n}$ we obtain directly from Lemma~\ref{reps:lem:1} the following decomposition:

\begin{prop}[{\cite[Theorem~6.4]{BenkartMoon}}]\label{reps:prop:1}
  The tensor powers of $L(\epsilon_1)$ decompose as
  \begin{equation}
    L(\epsilon_1)^{\otimes m} \cong \bigoplus_{\ell=0}^{m-1} \binom{m-1}{\ell}
    L(m\epsilon_1 - \ell \alpha).
  \end{equation}
\end{prop}

Let us now consider mixed tensor products, involving also the dual $L(\epsilon_1)^*$. By \eqref{reps:eq:17} we have that $L(\epsilon_1)^*$ is
isomorphic to $L(-\epsilon_2)$.
The following generalizes Proposition~\ref{reps:prop:1}:
\begin{theorem}
  \label{reps:thm:1}
  Suppose $m \neq n$. Then we have the following decomposition:
  \begin{equation}
    \label{reps:eq:16}
    L(\epsilon_1)^{\otimes m} \otimes L(\epsilon_1)^{* \otimes n} \cong
    \bigoplus_{\ell=0}^{m+n-1}
\binom{m+n-1}{\ell} L(m\epsilon_1-n\epsilon_2-\ell\alpha).
  \end{equation}
  On the other hand, we have
  \begin{equation}
    \label{eq:45}
    L(\epsilon_1)^{\otimes n } \otimes L(\epsilon_1)^{* \otimes n} \cong \bigoplus_{i=1}^{2^{\mathrlap{2n - 2}} } \big(L(\epsilon_1) \otimes L(\epsilon_1)^*\big)
  \end{equation}
  and $L(\epsilon_1) \otimes L(\epsilon_1)^*$ is indecomposable but not irreducible.
\end{theorem}

\begin{proof}
  The decomposition \eqref{reps:eq:16} follows from Lemma \ref{reps:lem:1} by induction. To obtain \eqref{eq:45} write $L(\epsilon_1)^{\otimes n} \otimes L(\epsilon_1)^{* \otimes n} = (L(\epsilon_1)^{\otimes n} \otimes L(\epsilon_1)^{* \otimes {n-1}}) \otimes L(\epsilon_1)^*$ and use \eqref{reps:eq:16} together with Lemma \ref{lem:2}.
\end{proof}

In particular, notice that $L(\epsilon_1)^{\otimes m} \otimes L(\epsilon_1)^{*\otimes n}$ is semisimple as long as $m \neq n$.

\section{Invariants of links}
\label{sec:invariant-links}

In this section we define the ribbon structure on the category of
representations of $\Uqgl$ and derive the corresponding invariants of
oriented framed tangles and links.

Recall that if $W$ is an $n$-dimensional complex super vector space the \emph{evaluation maps} are defined by
\begin{equation}
  \label{eq:8}
  \begin{aligned}
    \ev_W: W^* \otimes W & \longrightarrow \C(q), \qquad & \widehat \ev_W: W \otimes W^* & \longrightarrow \C(q),\\
    \phi \otimes w & \longmapsto \phi(w), & w \otimes \phi & \longmapsto (-1)^{\abs{\phi}\abs{w}} \phi(w),
  \end{aligned}
\end{equation}
and the \emph{coevalutaion maps} are defined by
\begin{equation}
\label{eq:9}
  \begin{aligned}
    \coev_W: \C(q) & \longrightarrow W \otimes W^*, &\qquad \widehat\coev_W: \C(q) & \longrightarrow W^* \otimes W,\\
   1 & \longmapsto \sum_{i=1}^n w_i \otimes w_i^*, & 1  & \longmapsto \sum_{i=1}^n (-1)^{\abs{w_i}} w_i^* \otimes w_i,
  \end{aligned}
\end{equation}
where $w_i$ is a basis of $W$ and $w_i^*$ is the corresponding dual basis of $W^*$. Note that if $\sigma_{V,W}$ denotes the map
\begin{equation}
  \label{eq:10}
  \begin{aligned}
    \sigma_{V,W} : V \otimes W &\longrightarrow W \otimes V\\
    v \otimes w & \longmapsto (-1)^{\abs{v}\abs{w}} w \otimes v.
  \end{aligned}
\end{equation}
then $\widehat \ev_W = \ev_W \circ \sigma_{W^*,W}$ and $\widehat \coev_W = \sigma_{W,W^*} \circ \coev_W$.

\subsection{Ribbon structure on $\Uqgl$-representations}
\label{sec:ribb-struct-uqgl}

Following the arguments of Tanisaki \cite{MR1187582} (see also \cite[\textsection{}10.1.D]{MR1300632}), we can construct a ribbon structure on the category of  $\Uqgl$-representations using the ribbon superalgebra structure on $\Uhgl$. We indicate now the main steps of those arguments.

The key observation is that, although $\Upsilon$ does not make sense
as an element of $\Uqgl \otimes \Uqgl$, it acts on every tensor
product $V \otimes W$ of two finite-dimensional
$\Uqgl$-modules. In other words, there is a well-defined
operator $\Upsilon_{V,W} \in \End_{\C(q)}(V \otimes W)$ determined by
setting $\Upsilon_{V,W} (v_\lambda \otimes w_\mu) = q^{(\lambda,\mu)} (v_\lambda \otimes w_\mu)$
if $v_\lambda$ and $w_\mu$ have weights $\lambda$ and $\mu$
respectively. Note however that $\Upsilon_{V,W}$ is not $\Uqgl$-equivariant, since $\Upsilon$ satisfies $\Upsilon \Delta(x) = \oDelta^\op(x) \Upsilon$ (see Lemma~\ref{lem:1}).

On the other hand, notice that the definition \eqref{reps:eq:34} of $\Theta$ makes sense also in $\Uqgl$, and \eqref{eq:5} holds in $\Uqgl$. Moreover, one has the following counterpart of equations \eqref{eq:37} and \eqref{eq:41}:
\begin{align}
  (\Delta \otimes \id) (\Theta) & = \Theta_{13} (\Upsilon_{V,Z})_{13} \Theta_{23} (\Upsilon_{V,Z}^{-1})_{13}\\
  (\id \otimes \Delta) (\Theta)  &= \Theta_{13} (\Upsilon_{V,Z})_{13} \Theta_{12} (\Upsilon_{V,Z}^{-1})_{13}.
\end{align}
This is now an equality of linear endomorphisms of $V \otimes W \otimes Z$ for all finite-dimensional $\Uqgl$-representations $V,W,Z$. Setting 
\begin{equation}
  \label{eq:46}
  R_{V,W} = \Theta \Upsilon_{V,W} \in \End_{\C(q)}(V \otimes W)
\end{equation}
one gets an operator which satisfies the Yang-Baxter equation. Note that $R_{V,W}$ is invertible, since $\Theta$ and $\Upsilon_{V,W}$ both are. Because of \eqref{reps:eq:5}, if we define $\check R_{V,W} = \sigma \circ R_{V,W}$, where $\sigma: V \otimes W \mapto W \otimes V$ is defined by $\sigma(v \otimes w) = (-1)^{\abs{v}{\abs{w}}} w \otimes v$, then we get an $\Uqgl$-equivariant isomorphism $\check R_{V,W} \in \Hom_{\Uqgl}(V \otimes W, W \otimes V)$.

Analogously, although the elements $u$ and $v$ do not make sense in $\Uqgl$, they act on each finite-dimensional $\Uqgl$-representation $V$ as operators $u_V, v_V \in \End_{\Uqgl}(V)$ (they are $\Uqgl$-equivariant because $u$, $v$ are central in $\Uhgl$). In the following, we will forget the subscripts of the operators $\check R$, $u$ and $v$.

For convenience, we give explicit formulas for the (inverse of the) operator $\check R_{L(\lambda),L(\mu)}$ for $\lambda, \mu \in \sfP'$:
\begin{equation}
\begin{aligned}
  \check R^{-1} (v^\lambda_1 \otimes v^\mu_1) & =
  (-1)^{(\abs{\lambda}+1)(\abs{\mu}+1)} q^{-(\mu-\alpha,\lambda-\alpha)} v^\mu_1
  \otimes v^\lambda_1,\\
  \check R^{-1} (v^\lambda_1 \otimes v^\mu_0) &=
  (-1)^{(\abs{\lambda}+1)\abs{\mu}} \big( q^{-(\mu,\lambda-\alpha)}v^\mu_0
  \otimes v^\lambda_1,\\
  & \phantom{=}+  (-1)^{\abs{\mu}} q^{-(\mu-\alpha,\lambda)}
  (q^{-1} -q)[\mu] v^\mu_1 \otimes v^\lambda_0\big) \\
  \check R^{-1} (v^\lambda_0 \otimes v^\mu_1) &=
  (-1)^{\abs{\lambda}(\abs{\mu}+1)} q^{-(\mu-\alpha,\lambda)}
  v^\mu_1 \otimes v^\lambda_0, \\
  \check R^{-1}(v^\lambda_0 \otimes v^\mu_0) &= (-1)^{\abs{\lambda}\abs{\mu}} q^{-(\mu,\lambda)} v^\mu_0 \otimes
  v^\lambda_0.
\end{aligned}\label{reps:eq:18}
\end{equation}

\subsection{Invariants of tangles}
\label{sec:invariant-tangles}

Let $D$ be an oriented framed tangle diagram. We will not draw the
framing because we will always suppose that it is the \emph{blackboard
  framing}. (Recall that a framing is a trivialization of the normal
bundle: since the tangle is oriented, such a trivialization is
uniquely determined by a section of the normal bundle; the blackboard
framing is the trivialization determined by the unit vector orthogonal to
the plane -- or to the blackboard -- pointing outwards.)

 We assume $D \subset \R \times [0,1]$ and we let $\bolds(D) = D \cap (\R \times 0)= \{s^D_1,\ldots,s^D_a\}$ with $s^D_1 < \cdots < s^D_a$ be the source points of $D$ and $\boldt(D)=D \cap (\R \times 1) = \{t^D_1,\ldots,t^D_b\}$ with $t^D_1 < \cdots < t^D_b$ be the target points of $D$. Let also $\boldell$ be a labeling of the strands of $D$ by simple two-dimensional representations of $\Uqgl$ (that is, a map from the set of strands of $D$ to $\sfP'$). We indicate by $\ell^s_1,\ldots,\ell^s_a$ the labeling of the strands at the source points of $D$ and by $\ell^t_1,\ldots,\ell^t_b$ the labeling at the target points. Moreover, we let $\gamma_1^s,\ldots,\gamma_a^s$ and $\gamma_1^t,\ldots,\gamma_b^t$ be the signs corresponding to the orientations of the strands at the source and target points (where $+1$ corresponds to a strand oriented upwards and $-1$ to a strand oriented downwards). 

Given these data, one can define a $\Uqgl$-equivariant map
\begin{equation}
  \label{eq:18}
  Q^\boldell(D) : L(\ell^s_1)^{\gamma^s_1} \otimes \cdots \otimes L(\ell^s_a)^{\gamma^s_a} \longrightarrow L(\ell^t_1)^{\gamma^t_1} \otimes \cdots \otimes L(\ell^t_b)^{\gamma^t_b},
\end{equation}
where $L(\lambda)^{-1}=L(\lambda)^*$, by decomposing $D$ into elementary pieces as shown below and assigning the corresponding morphisms as displayed.\label{diagrammi}
\begin{align*}
  \label{eq:7}
 & Q\left(
    \begin{tikzpicture}[anchorbase]
      \draw[->,very thick] (0,0) -- node[left] {$V$} (0,1);
    \end{tikzpicture}
\;\;\; \right)  =
\begin{tikzpicture}[anchorbase]
  \draw[->] (0,0) node[below] {$V$} -- node[left] {$\id$} (0,1) node[above] {$V$};
\end{tikzpicture} 
  &Q\left(
    \begin{tikzpicture}[anchorbase]
      \draw[<-,very thick] (0,0) -- node[left] {$V$} (0,1);
    \end{tikzpicture}
\;\;\; \right)  =
\begin{tikzpicture}[anchorbase]
  \draw[->] (0,0) node[below] {$V^*$} -- node[left] {$\id$} (0,1) node[above] {$V^*$};
  \node at (0,0) {\hphantom{$V \otimes V^*$}};
\end{tikzpicture}&\\
  &Q\left(
    \begin{tikzpicture}[anchorbase]
      \draw[->,very thick] (1,0) -- node[right,near start] {$W$} (0,2);
      \draw[->,very thick,cross line] (0,0) -- node[left,near start] {$V$} (1,2);
    \end{tikzpicture}
\right)  =
\begin{tikzpicture}[anchorbase]
  \draw[->] (0,0) node[below] {$V \otimes W$} -- node[left] {$\check R$} (0,1) node[above] {$W \otimes V$};
\end{tikzpicture} 
 &Q\left(
    \begin{tikzpicture}[anchorbase]
      \draw[->,very thick] (0,0) -- node[left,near start] {$V$} (1,2);
      \draw[->,very thick,cross line] (1,0) -- node[right,near start] {$W$} (0,2);
    \end{tikzpicture}
\right)  =
\begin{tikzpicture}[anchorbase]
  \draw[->] (0,0) node[below] {$V \otimes W$} -- node[left] {$\check R^{-1}$} (0,1) node[above] {$W \otimes V$};
\end{tikzpicture}&\\
  &Q\left(
    \begin{tikzpicture}[anchorbase]
      \draw[->,very thick] (0,0)  arc (180:0:0.5 and 0.75) node[right,midway] {$V$};
    \end{tikzpicture}
\right)  =
\begin{tikzpicture}[anchorbase]
  \draw[->] (0,0) node[below] {$V \otimes V^*$} -- node[left] {$\widehat\ev \circ (uv^{-1} \otimes \id)$} (0,1) node[above] {$\C$};
\end{tikzpicture}
  &Q\left(
    \begin{tikzpicture}[anchorbase]
      \draw[->,very thick] (0,0)  arc (0:180:0.5 and 0.75) node[left,midway] {$V$};
    \end{tikzpicture}
\right)  =
\begin{tikzpicture}[anchorbase]
  \draw[->] (0,0) node[below] {$V^* \otimes V$} -- node[left] {$\ev$} (0,1) node[above] {$\C$};
\end{tikzpicture}&\\
  &Q\left(
    \begin{tikzpicture}[anchorbase]
      \draw[<-,very thick] (0,0)  arc (360:180:0.5 and 0.75) node[left,midway] {$V$};
    \end{tikzpicture}
\right)  =
\begin{tikzpicture}[anchorbase]
  \draw[->] (0,0) node[below] {$\C$} -- node[left] {$(\id \otimes vu^{-1} ) \circ \widehat\coev $} (0,1) node[above] {$V^* \otimes V$};
\end{tikzpicture}
  &Q\left(
    \begin{tikzpicture}[anchorbase]
      \draw[<-,very thick] (0,0)  arc (180:360:0.5 and 0.75) node[right,midway] {$V$};
    \end{tikzpicture}
\right)  =
\begin{tikzpicture}[anchorbase]
  \draw[->] (0,0) node[below] {$\C$} -- node[left] {$\coev$} (0,1) node[above] {$V \otimes V^*$};
\end{tikzpicture}&\end{align*}

As we already mentioned, although $\Uqgl$ itself is not a ribbon superalgebra, its representation category is a ribbon category. Hence we have:

\begin{theorem}
  \label{thm:1}
  The map $Q^\boldell(D)$ just defined is an isotopy invariant of oriented framed tangles.
\end{theorem}

The proof, for which we refer to \cite[Theorem 4.7]{MR1881401}, is a direct check of the Reidemeister moves (or, more precisely, of the analogues of the Reidemeister moves for framed tangles). In fact, the axioms of a ribbon category are equivalent to the validity of these moves.

If all strands are labeled by the same simple representation $L(\lambda)$ (i.e.\ $\boldell$ is the constant map with value $\lambda$), then we write $Q^\lambda(D)$ instead of $Q^\boldell(D)$.

Let us indicate a full $+1$ twist by the symbol
\begin{equation}
  \label{eq:11}
  \begin{tikzpicture}[anchorbase]
    \draw[->,very thick] (0,0) -- node[draw, fill=white, text=black] {$1$} ++(0,2);
  \end{tikzpicture}
\;\;=\;    \begin{tikzpicture}[anchorbase]
      \draw[<-,very thick] (0,2) .. controls +(-80:1.5cm) and +(down:1cm) .. (1,1);
      \draw[very thick,cross line] (0,0) .. controls +(80:1.5cm) and +(up:1cm) .. (1,1);
  \end{tikzpicture}
\end{equation}
Then we have (cf.\ \cite[\textsection{}4.2]{MR1881401})
\begin{equation*}
Q\left(  \begin{tikzpicture}[anchorbase]
    \draw[->,very thick] (0,0) -- node[left, pos=0.2] {$V$} node[draw, fill=white, text=black] {$1$} ++(0,1.6);
  \end{tikzpicture}
\right) = \begin{tikzpicture}[anchorbase]
  \draw[->] (0,0) node[below] {$V$} -- node[left] {$v$} (0,1) node[above] {$V$};
\end{tikzpicture} 
\end{equation*}

\begin{lemma}
  \label{lem:4}
  The element $v$ acts by the identity on the vector representation $L(\epsilon_1)$ and on its dual $L(\epsilon_1)^*$.
\end{lemma}
\begin{proof}
  Recall that we denote by $v^{\epsilon_1}_0$, $v^{\epsilon_1}_1$ the standard basis of $L(\epsilon_1)$. We have
  \begin{equation}
    \label{eq:12}
    \begin{aligned}
      v v_0^{\epsilon_1} & = (K^{-1} + (q-q^{-1})EF) \quantumq^{- (h_1+h_2)(h_1-h_2)}  v_0^{\epsilon_1}\\
      & = (K^{-1} + (q-q^{-1})EF) q^{-\langle h_1+h_2, \epsilon_1\rangle \langle h_1-h_2, \epsilon_1 \rangle} v_0^{\epsilon_1}\\
        & = (q^{-1} + q - q^{-1}) q^{-1} v_0^{\epsilon_1} = v_0^{\epsilon_1}.
    \end{aligned}
  \end{equation}
  Since $L(\epsilon_1)$ is irreducible and $v$ acts in an $\Uqgl$-equivariant way, it follows that $v$ acts by the identity on $L(\epsilon_1)$. Since $S(v)=v$, the element $v$ acts by the identity also on $L(\epsilon_1)^*$.
\end{proof}

As a consequence, if we label all strands of our tangles by the vector representation then we need not worry about the framing any more:

\begin{corollary}
  \label{cor:1}
  The assignment $D \mapsto Q^{\epsilon_1}(D)$ is an invariant of oriented tangles.
\end{corollary}

\subsection{Invariants of links}
\label{sec:invariants-links}

Since links are in particular tangles, we obtain from $Q^\boldell$ an invariant of oriented framed links;
unfortunately, this invariant is always zero:

\begin{prop}
  \label{prop:3}
  Let $L$ be a closed link diagram and $\boldell$ a labeling of its strands. Then $Q^\boldell(L)=0$.
\end{prop}

\begin{proof}
  The invariant associated to $L$ is some endomorphism $\phi$ of the
  trivial representation $\C(q)$. Up to isotopy, we can assume that
  there is some level at which the link diagram $L$ has only two
  strands, one oriented upwards and the other one downwards, labeled
  by the same weight $\lambda$. Without loss of generality suppose
  that the leftmost is oriented upwards. Slice the diagram at this
  level, so that we can write $\phi$ as the composition $\phi_2 \circ
  \phi_1$ of two $\Uqgl$-equivariant maps $\phi_1: \C(q) \mapto
  L(\lambda) \otimes L(\lambda)^* $ and $\phi_2 :L(\lambda) \otimes
  L(\lambda)^* \mapto \C(q)$. If $\phi = \phi_2 \circ \phi_1$ is not
  zero, then we have an inclusion $\phi_1$ of $\C(q)$ inside
  $L(\lambda) \otimes L(\lambda)^*$ and a projection $\phi_2$ of the
  latter onto $\C(q)$, so that $\C(q)$ would be a direct summand of
  $L(\lambda) \otimes L(\lambda)^*$. But this is not possible, since
  $L(\lambda) \otimes L(\lambda)^*$ is indecomposable (by Lemma \ref{lem:2}); hence $\phi=0$.
\end{proof}

To get invariants of closed links we need to cut the links, as we are going to explain now.
First, we need the following result:

\begin{prop}
  \label{prop:5}
  Let $D$ be an oriented tangle diagram with two source points and two target points. Let $\boldell$ be a labeling of the strands of $D$ such that $\ell^s_1=\ell^s_2=\ell^t_1=\ell^t_2$. Then
  \begin{equation}
    \label{eq:20}
Q^\boldell \left(  \begin{tikzpicture}[anchorbase]
    \draw[very thick, postaction={decorate}, decoration={markings, mark=at position 0.3 with {\arrow{>}}}] (0,0.4) -- ++(0,0.7) arc (0:180:0.3 and 0.5) -- ++(0,-0.7) arc (180:360:0.3 and 0.5) -- cycle ;
    \draw[->,very thick] (0.3,0) -- ++(0,1.5);
    \node[draw, fill=white, text=black, minimum width=0.8cm] at (0.15,0.75) {$D$};
  \end{tikzpicture}\right) =
Q^\boldell \left(  \begin{tikzpicture}[anchorbase,x=-1cm]
    \draw[very thick, postaction={decorate}, decoration={markings, mark=at position 0.3 with {\arrow{>}}}] (0,0.4) -- ++(0,0.7) arc (0:180:0.3 and 0.5) -- ++(0,-0.7) arc (180:360:0.3 and 0.5) -- cycle ;
    \draw[->,very thick] (0.3,0) -- ++(0,1.5);
    \node[draw, fill=white, text=black, minimum width=0.8cm] at (0.15,0.75) {$D$};
  \end{tikzpicture}\right)
  \end{equation}
\end{prop}

\begin{proof}
  Let $\ell_1^s=\lambda$. Then $Q^\boldell(D)= \phi$ where $\phi:
  L(\lambda) \otimes L(\lambda) \mapto L(\lambda) \otimes
  L(\lambda)$. By Lemma \ref{reps:lem:1} the representation
  $L(\lambda) \otimes L(\lambda)$ is isomorphic to the direct sum $L(2
  \lambda) \oplus L(2 \lambda - \alpha)$. Let $e_1,e_2$ be the two orthogonal idempotents corresponding to this decomposition.

  We consider formal $\C(q)$-linear combinations of tangle diagrams,
  and we extend $Q^\boldell$ to them. Since $\End_{\Uqgl}(L(\lambda) \otimes L(\lambda))$ is a two-dimensional $\C(q)$-vector space and $\check R_{\lambda,\lambda}$ is not a multiple of the identity by \eqref{reps:eq:18}, there are
  some $\C(q)$-linear combinations of tangle diagrams $E_1$ and $E_2$
  such that $Q^\boldell(E_1) = e_1$ and $Q^\boldell(E_2)=e_2$. Hence we
  can write
  \begin{equation}
    \label{eq:25}
Q^\boldell \left(  \begin{tikzpicture}[anchorbase]
    \draw[->,very thick] (0,0) -- ++(0,1.5);
    \draw[->,very thick] (0.3,0) -- ++(0,1.5);
    \node[draw, fill=white, text=black, minimum width=0.8cm] at (0.15,0.75) {$E_1$};
  \end{tikzpicture}\right)
+ Q^\boldell \left(  \begin{tikzpicture}[anchorbase]
    \draw[->,very thick] (0,0) -- ++(0,1.5);
    \draw[->,very thick] (0.3,0) -- ++(0,1.5);
    \node[draw, fill=white, text=black, minimum width=0.8cm] at (0.15,0.75) {$E_2$};
  \end{tikzpicture}\right) =
Q^\boldell \left(\,  \begin{tikzpicture}[anchorbase]
    \draw[->,very thick] (0,0) -- ++(0,1.5);
    \draw[->,very thick] (0.3,0) -- ++(0,1.5);
  \end{tikzpicture}\,\right).
  \end{equation}

Now we have
\begin{equation}
  \label{eq:26}
  \begin{aligned}
  Q^\boldell \left(  \begin{tikzpicture}[anchorbase]
    \draw[very thick, postaction={decorate}, decoration={markings, mark=at position 0.3 with {\arrow{>}}}] (0,0.4) -- ++(0,0.7) arc (0:180:0.3 and 0.5) -- ++(0,-0.7) arc (180:360:0.3 and 0.5) -- cycle ;
    \draw[->,very thick] (0.3,0) -- ++(0,1.5);
    \node[draw, fill=white, text=black, minimum width=0.8cm] at (0.15,0.75) {$D$};
  \end{tikzpicture}\right) & =
  Q^\boldell \left(  \begin{tikzpicture}[anchorbase]
    \draw[very thick, postaction={decorate}, decoration={markings, mark=at position 0.3 with {\arrow{>}}}] (0,0.4) -- ++(0,0.7) arc (0:180:0.3 and 0.5) -- ++(0,-0.7) arc (180:360:0.3 and 0.5) -- cycle ;
    \draw[->,very thick] (0.3,0) -- ++(0,1.5);
    \node[draw, fill=white, text=black, minimum width=0.8cm] at (0.15,1) {$D$};
    \node[draw, fill=white, text=black, minimum width=0.8cm] at (0.15,0.4) {$E_1$};
  \end{tikzpicture}\right) +
  Q^\boldell \left(  \begin{tikzpicture}[anchorbase]
    \draw[very thick, postaction={decorate}, decoration={markings, mark=at position 0.3 with {\arrow{>}}}] (0,0.4) -- ++(0,0.7) arc (0:180:0.3 and 0.5) -- ++(0,-0.7) arc (180:360:0.3 and 0.5) -- cycle ;
    \draw[->,very thick] (0.3,0) -- ++(0,1.5);
    \node[draw, fill=white, text=black, minimum width=0.8cm] at (0.15,1) {$D$};
    \node[draw, fill=white, text=black, minimum width=0.8cm] at (0.15,0.4) {$E_2$};
  \end{tikzpicture}\right)\\
&=   Q^\boldell \left(  \begin{tikzpicture}[anchorbase]
    \draw[very thick, postaction={decorate}, decoration={markings, mark=at position 0.4 with {\arrow{>}}}] (0.3,0) -- ++(0,1.3) arc (0:180:0.45 and 0.5) -- ++(0,-1.3) arc (180:360:0.45 and 0.5) -- cycle;
    \draw[->,cross line, very thick] (0.3,-0.5) .. controls +(0,0.2) and +(0,-0.2) .. (0,0) -- ++(0,1.4) .. controls +(0,0.2) and +(0,-0.2) .. ++(0.3,0.5);
    \node[draw, fill=white, text=black, minimum width=0.8cm] at (0.15,1) {$D$};
    \node[draw, fill=white, text=black, minimum width=0.8cm] at (0.15,0.4) {$E_1$};
  \end{tikzpicture}\right) +
  Q^\boldell \left(  \begin{tikzpicture}[anchorbase]
    \draw[very thick, postaction={decorate}, decoration={markings, mark=at position 0.4 with {\arrow{>}}}] (0.3,0) -- ++(0,1.3) arc (0:180:0.45 and 0.5) -- ++(0,-1.3) arc (180:360:0.45 and 0.5) -- cycle;
    \draw[->,cross line, very thick] (0.3,-0.5) .. controls +(0,0.2) and +(0,-0.2) .. (0,0) -- ++(0,1.4) .. controls +(0,0.2) and +(0,-0.2) .. ++(0.3,0.5);
    \node[draw, fill=white, text=black, minimum width=0.8cm] at (0.15,1) {$D$};
    \node[draw, fill=white, text=black, minimum width=0.8cm] at (0.15,0.4) {$E_2$};
  \end{tikzpicture}\right) \\
& =
  Q^\boldell \left(  \begin{tikzpicture}[anchorbase,x=-1cm]
    \draw[very thick, postaction={decorate}, decoration={markings, mark=at position 0.3 with {\arrow{>}}}] (0,0.4) -- ++(0,0.7) arc (0:180:0.3 and 0.5) -- ++(0,-0.7) arc (180:360:0.3 and 0.5) -- cycle ;
    \draw[->,very thick] (0.3,0) -- ++(0,1.5);
    \node[draw, fill=white, text=black, minimum width=0.8cm] at (0.15,1) {$D$};
    \node[draw, fill=white, text=black, minimum width=0.8cm] at (0.15,0.4) {$E_1$};
  \end{tikzpicture}\right) +
  Q^\boldell \left(  \begin{tikzpicture}[anchorbase,x=-1cm]
    \draw[very thick, postaction={decorate}, decoration={markings, mark=at position 0.3 with {\arrow{>}}}] (0,0.4) -- ++(0,0.7) arc (0:180:0.3 and 0.5) -- ++(0,-0.7) arc (180:360:0.3 and 0.5) -- cycle ;
    \draw[->,very thick] (0.3,0) -- ++(0,1.5);
    \node[draw, fill=white, text=black, minimum width=0.8cm] at (0.15,1) {$D$};
    \node[draw, fill=white, text=black, minimum width=0.8cm] at (0.15,0.4) {$E_2$};
  \end{tikzpicture}\right) =
   Q^\boldell \left(  \begin{tikzpicture}[anchorbase,x=-1cm]
    \draw[very thick, postaction={decorate}, decoration={markings, mark=at position 0.3 with {\arrow{>}}}] (0,0.4) -- ++(0,0.7) arc (0:180:0.3 and 0.5) -- ++(0,-0.7) arc (180:360:0.3 and 0.5) -- cycle ;
    \draw[->,very thick] (0.3,0) -- ++(0,1.5);
    \node[draw, fill=white, text=black, minimum width=0.8cm] at (0.15,0.75) {$D$};
  \end{tikzpicture}\right).
\end{aligned}
\end{equation}
The second equality here follows because we must have 
  \begin{equation}
    \label{eq:24}
    \check R e_1 = e_1 \check R = a_1 e_1 \qquad \text{and} \qquad \check R e_2 =  e_2 \check R  = a_2 e_2
  \end{equation}
for some $a_1,a_2 \in \C(q)$, since $e_1$ and $e_2$ project onto one-dimensional subspaces of $\End_{\Uqgl}(L(\lambda) \otimes L(\lambda))$. The penultimate equality follows by isotopy invariance.
\end{proof}

Let now $D$ be an oriented framed link diagram, $\boldell$ a labeling of its strands and $\lambda \in \sfP'$ some weight which labels some strand of $D$. By cutting one of the strands labeled by $\lambda$, we can suppose that $D$ is the closure of a tangle $\tilde D$ with one source and one target point, as in the picture
\begin{equation}
  \label{eq:27}
  \begin{tikzpicture}[anchorbase,x=-1cm]
    \node[draw, fill=white, text=black, minimum width=0.7cm, minimum height=1cm] at (0,0.75) {$D$};
  \end{tikzpicture} =
  \begin{tikzpicture}[anchorbase,x=-1cm]
    \draw[very thick, postaction={decorate}, decoration={markings, mark=at position 0.3 with {\arrow{>}}}] (0,0.4) -- ++(0,0.7) arc (0:180:0.3 and 0.5) node[right] {$L(\lambda)$} -- ++(0,-0.7) arc (180:360:0.3 and 0.5) -- cycle ;
    \node[draw, fill=white, text=black, minimum width=0.7cm] at (0,0.75) {$\tilde D$};
  \end{tikzpicture}
\end{equation}
Then we define $\hat Q^{\boldell,\lambda}(D) = c \in \C(q)$ where
\begin{equation}
  \label{eq:28}
Q^{\boldell}\left(  \begin{tikzpicture}[anchorbase,x=-1cm]
    \draw[very thick,->] (0,0) -- ++(0,1.5) node[yshift=-0.2cm,right] {$L(\lambda)$};
    \node[draw, fill=white, text=black, minimum width=0.7cm] at (0,0.75) {$\tilde D$};
  \end{tikzpicture}\right) = c \cdot \id_{L(\lambda)}
\end{equation}
We have:
\begin{theorem}
  \label{thm:2}
  The assignment $D \mapsto\hat Q^{\boldell,\lambda}(D) \in \C(q)$ is an invariant of oriented framed links.
\end{theorem}
\begin{proof}
  Since $Q^\boldell(\tilde D)$ is an invariant of oriented framed tangles, we need only to show that $\hat Q^{\boldell,\lambda}$ is independent of how we cut $D$ to get $\tilde D$. If $\tilde D'$ is obtained by some different cutting, but always along some strand labeled by $\lambda$, then after some isotopy we must have
  \begin{equation}
    \label{eq:29}
\begin{tikzpicture}[anchorbase,x=-1cm]
    \draw[very thick,->] (0,0) -- ++(0,1.5) node[yshift=-0.2cm,right] {$L(\lambda)$};
    \node[draw, fill=white, text=black, minimum width=0.7cm] at (0,0.75) {$\tilde D$};
  \end{tikzpicture} =  \begin{tikzpicture}[anchorbase]
    \draw[very thick, postaction={decorate}, decoration={markings, mark=at position 0.3 with {\arrow{>}}}] (0,0.4) -- ++(0,0.7) arc (0:180:0.3 and 0.3) -- ++(0,-0.7) arc (180:360:0.3 and 0.3) -- cycle ;
    \draw[->,very thick] (0.3,0) -- ++(0,1.5);
    \node[draw, fill=white, text=black, minimum width=0.8cm] at (0.15,0.75) {$D^{(2)}$};
  \end{tikzpicture} \qquad \text{and} \qquad
\begin{tikzpicture}[anchorbase,x=-1cm]
    \draw[very thick,->] (0,0) -- ++(0,1.5) node[yshift=-0.2cm,right] {$L(\lambda)$};
    \node[draw, fill=white, text=black, minimum width=0.7cm] at (0,0.75) {$\tilde D'$};
  \end{tikzpicture} =  \begin{tikzpicture}[anchorbase,x=-1cm]
    \draw[very thick, postaction={decorate}, decoration={markings, mark=at position 0.3 with {\arrow{>}}}] (0,0.4) -- ++(0,0.7) arc (0:180:0.3 and 0.3) -- ++(0,-0.7) arc (180:360:0.3 and 0.3) -- cycle ;
    \draw[->,very thick] (0.3,0) -- ++(0,1.5);
    \node[draw, fill=white, text=black, minimum width=0.8cm] at (0.15,0.75) {$D^{(2)}$};
  \end{tikzpicture}
  \end{equation}
for some tangle $D^{(2)}$. By Proposition \ref{prop:5} we have then $Q^{\boldell}(\tilde D)= Q^{\boldell} (\tilde D')$.
\end{proof}

If $\boldell$ is the constant labeling by the weight $\lambda$, we write $\hat Q^\lambda$ instead of $\hat Q^{\boldell,\lambda}$. For $\lambda=\epsilon_1$ we write simply $\hat Q$. As a consequence of Corollary \ref{cor:1} and Theorem \ref{thm:2} we obtain:

\begin{corollary}
  \label{cor:2}
  The assignment $D \mapsto \hat Q(D) \in \C(q)$ is an invariant of oriented links.
\end{corollary}

\subsection{Recovering the Alexander polynomial}
\label{sec:recov-alex-polyn}

If we compute the action of the $R$-matrix on $L(\epsilon_1) \otimes L(\epsilon_1)$ we get by \eqref{reps:eq:18}, setting $v_1=v_1^{\epsilon_1}$ and $v_0=v_0^{\epsilon_1}$:
\begin{equation}
\begin{aligned}
  \check R^{-1} (v_1 \otimes v_1) & = -q v_1 \otimes v_1, & \check R^{-1} (v_1 \otimes v_0) &= v_0 \otimes v_1 +(q^{-1} -q) v_1 \otimes v_0 \\
  \check R^{-1} (v_0 \otimes v_1) &= v_1 \otimes v_0 & \check R^{-1}(v_0 \otimes v_0) &= q^{-1} v_0 \otimes
  v_0.
\end{aligned}\label{eq:30}
\end{equation}
On can easily check that
\begin{equation}
(\check R^{-1})^2 = (q^{-1}-q)\check R^{-1} + \text{Id}.\label{reps:eq:20}
\end{equation}
and hence
\begin{equation}
  \label{eq:31}
  \check R = \check R^{-1} + (q - q^{-1}) \text{Id}.
\end{equation}

It follows:
\begin{prop}
  \label{prop:6}
  The invariant of links $\hat Q$ satisfies the following skein relation
  \begin{equation}
    \label{eq:32}
  \hat Q\left(
    \begin{tikzpicture}[anchorbase]
      \draw[->,very thick] (1,0) -- (0,1);
      \draw[->,very thick,cross line] (0,0) -- (1,1);
      \draw[dashed] (0.5,0.5) circle (0.73);
    \end{tikzpicture}
\right) -
  \hat Q\left(
    \begin{tikzpicture}[anchorbase]
      \draw[->,very thick] (0,0) -- (1,1);
      \draw[->,very thick,cross line] (1,0) -- (0,1);
      \draw[dashed] (0.5,0.5) circle (0.73);
    \end{tikzpicture}
\right) = 
  (q-q^{-1}) \cdot \hat Q\left(
    \begin{tikzpicture}[anchorbase]
      \draw[->,very thick] (0,0) .. controls +(0.5,0.5) and +(0.5,-0.5) .. (0,1);
      \draw[->,very thick] (1,0) .. controls +(-0.5,0.5) and +(-0.5,-0.5) .. (1,1);
      \draw[dashed] (0.5,0.5) circle (0.73);
    \end{tikzpicture}
\right)  \end{equation}
where the pictures represent three links that differ only inside a small neighborhood of a crossing.
\end{prop}

We recall one of the equivalent definitions of the Alexander-Conway polynomial (\cite{MR1501429}, \cite{MR0258014}):

\begin{definition}
  \label{def:1}
  The \emph{Alexander-Conway polynomial} is the value of the assignment
  \begin{equation}
\Delta: \mathrm{Links} \mapto \Z[t^{\frac{1}{2}},t^{-\frac{1}{2}}]\label{eq:34}
\end{equation}
 defined by the following skein relations:
  \begin{gather}
  \Delta\left(
    \begin{tikzpicture}[anchorbase]
      \draw[very thick] (0.5,0.5) circle (0.3);
    \end{tikzpicture}
\right) = 1,    \label{eq:33}\\[1ex]
  \Delta\left(
    \begin{tikzpicture}[anchorbase,scale=0.5]
      \draw[->,very thick] (1,0) -- (0,1);
      \draw[->,very thick,cross line] (0,0) -- (1,1);
      \draw[dashed] (0.5,0.5) circle (0.73);
    \end{tikzpicture}
\right) -
  \Delta\left(
    \begin{tikzpicture}[anchorbase,scale=0.5]
      \draw[->,very thick] (0,0) -- (1,1);
      \draw[->,very thick,cross line] (1,0) -- (0,1);
      \draw[dashed] (0.5,0.5) circle (0.73);
    \end{tikzpicture}
\right) = 
  (t^{\frac{1}{2}}-t^{-\frac{1}{2}}) \cdot \Delta\left(
    \begin{tikzpicture}[anchorbase,scale=0.5]
      \draw[->,very thick] (0,0) .. controls +(0.5,0.5) and +(0.5,-0.5) .. (0,1);
      \draw[->,very thick] (1,0) .. controls +(-0.5,0.5) and +(-0.5,-0.5) .. (1,1);
      \draw[dashed] (0.5,0.5) circle (0.73);
    \end{tikzpicture}
\right).
  \end{gather}
\end{definition}

Notice that obviously $\hat Q\big(\begin{tikzpicture}[baseline={([yshift=-0.7ex]current bounding box.center)}]    \draw[very thick] (0.15,0) circle (0.15);\end{tikzpicture}\big ) = 1$ since $Q^{\epsilon_1}\big(\;\begin{tikzpicture}[baseline={([yshift=-0.7ex]current bounding box.center)} ] \draw[very thick,->] (0,0) -- (0,0.5); \end{tikzpicture} \;\big) = \Id$. As a consequence, we have that $Q$ is essentially the Alexander-Conway polynomial:

\begin{theorem}
  \label{thm:3}
  For all oriented links $L$ in $\R^3$ we have
  \begin{equation}
    \label{eq:35}
    \Delta(L) = \hat Q(L)\raisebox{-0.5ex}{$\big|$}_{q=t^{\frac{1}{2}}}. 
  \end{equation}
  In particular, $\hat Q(L)\in \Z[q,q^{-1}]$ is a Laurent polynomial in $q$.
\end{theorem}

\section{Appendix}
\label{sec:proof}

We collect here two technical lemmas which are used in Section \ref{sec:quant-envel-super} to construct the ribbon structure on $U_\hbar$. Both result follows from the explicit construction of the $R$-matrix (cf.\ \cite{MR1134942}); we include the two proofs for completeness, although they are easy calculations.

\begin{lemma}
  \label{lem:1}
The following properties hold for all $x \in \Uhgl$:
\begin{align}
     \Theta  \overline\Delta^\op(x)& = \Delta^\op(x) \Theta\label{eq:5}\\
    \Upsilon  \Delta (x) & = \overline\Delta^\op (x) \Upsilon.\label{eq:6}
\end{align}
\end{lemma}

\begin{proof}
  It is enough to check \eqref{eq:5} and \eqref{eq:6} on the generators. We have
  \begin{align*}
    \Theta \overline \Delta^\op(E) & = \Theta ( K \otimes E + E \otimes 1)\\
    & = K \otimes E + E \otimes 1 + (q-q^{-1}) FK \otimes E^2 - (q-q^{-1}) FE \otimes E\\
    & = K \otimes E + E \otimes 1 + (q-q^{-1}) EF \otimes E - (K-K^{-1}) \otimes E\\
  & = K^{-1} \otimes E + E \otimes 1   + (q-q^{-1}) EF \otimes E\\
& = (K^{-1} \otimes E + E \otimes 1)\Theta  = \Delta^\op(E) \Theta
  \end{align*}
  and
  \begin{align*}
    \Theta \overline \Delta^\op(F) & = \Theta ( 1 \otimes F + F \otimes K^{-1})\\
    & = 1 \otimes F + F \otimes K^{-1} + (q-q^{-1}) F \otimes EF - (q-q^{-1}) F^2 \otimes EK^{-1}\\
    & = 1 \otimes F + F \otimes K^{-1} - (q-q^{-1}) F \otimes FE + F \otimes (K-K^{-1}) \\
    & = 1 \otimes F + F \otimes K - (q-q^{-1}) F \otimes FE \\
    & = (1 \otimes F + F \otimes K)\Theta = \Delta^\op(F) \Theta
  \end{align*}
and for $i=1,2$
  \begin{align*}
    \Theta \overline \Delta^\op(H_i) & = \Theta ( 1 \otimes H_i + H_i \otimes 1)\\
    & = 1 \otimes H_i + H_i \otimes 1 + (q-q^{-1}) F \otimes EH_i + (q-q^{-1}) F H_i \otimes E \\
    & = 1 \otimes H_i + H_i \otimes 1 - (q-q^{-1})\langle H_i, \alpha \rangle  F \otimes E + (q-q^{-1}) F \otimes H_i E+\\
    & \qquad \qquad + (q-q^{-1})\langle H_i, \alpha \rangle F  \otimes E + (q-q^{-1}) H_i F \otimes E\\
    & = 1 \otimes H_i + H_i \otimes 1 + (q-q^{-1}) F \otimes H_iE + (q-q^{-1}) H_i F \otimes E \\
    & = (1 \otimes H_i + H_i \otimes 1)\Theta = \Delta^\op(H_i) \Theta.
  \end{align*}
  Moreover, we have
  \begin{align*}
    \Upsilon \Delta(E) &= e^{\hbar (H_1 \otimes H_1 - H_2 \otimes H_2)} (E \otimes K^{-1} + 1 \otimes E) \\
    & = (E \otimes K^{-1}) e^{\hbar ((H_1+1) \otimes H_1 - (H_2-1) \otimes H_2)} + (1 \otimes E) e^{\hbar (H_1 \otimes (H_1+1) - H_2 \otimes (H_2-1))} \\
    & = (E \otimes 1 + K \otimes E) e^{\hbar (H_1 \otimes H_1 - H_2 \otimes H_2)} = \overline\Delta^{\mathrm{op}} (E) \Upsilon
  \end{align*}
and
  \begin{align*}
    \Upsilon \Delta(F) &= e^{\hbar (H_1 \otimes H_1 - H_2 \otimes H_2)} (F \otimes 1 + K \otimes F) \\
    & = (F \otimes 1) e^{\hbar ((H_1-1) \otimes H_1 - (H_2+1) \otimes H_2)} + (K \otimes F) e^{\hbar (H_1 \otimes (H_1-1) - H_2 \otimes (H_2+1))} \\
    & = (F \otimes K^{-1} + 1 \otimes F) e^{\hbar (H_1 \otimes H_1 - H_2 \otimes H_2)} = \overline\Delta^{\mathrm{op}} (F) \Upsilon.
  \end{align*}
Finally, for $i=1,2$ we have $\Upsilon \Delta(H_i) = \Delta(H_i) \Upsilon$ since the elements $H_1$, $H_2$ commute with each other. Since $\overline\Delta^{\mathrm{op}}(H_i) =\Delta(H_i)$ we get $\Upsilon \Delta(H_i) = \overline \Delta^{\mathrm{op}}(H_i) \Upsilon$ and we are done.
\end{proof}

\begin{lemma}
  \label{lem:5}
  In $\Uhgl$ the following identities hold:
  \begin{align}
    \label{eq:37}
    (\Delta \otimes \id) (\Theta) & = \Theta_{13} \Upsilon_{13} \Theta_{23} \Upsilon_{13}^{-1},\\
    (\id \otimes \Delta) (\Theta) & = \Theta_{13} \Upsilon_{13} \Theta_{12} \Upsilon_{13}^{-1}.\label{eq:41}
  \end{align}
\end{lemma}

\begin{proof}
  The two computations are similar, so let us check \eqref{eq:37} and leave \eqref{eq:41} to the reader. The l.h.s.\ is simply
  \begin{equation}
    \label{eq:38}
    (\Delta \otimes \id) (\Theta) = 1 + (q-q)^{-1} F \otimes 1 \otimes E + (q-q^{-1}) K \otimes F \otimes E.
  \end{equation}
  We will now compute the r.h.s. First we have
  \begin{equation}
    \label{eq:39}
    \begin{aligned}
      \Upsilon_{13} (1 \otimes F \otimes E) \Upsilon_{13}^{-1} & = \Upsilon_{13} (1 \otimes F \otimes E) e^{- \hbar (H_1 \otimes 1 \otimes H_1)} e^{\hbar( H_2 \otimes 1 \otimes H_2)} \\
      &= \Upsilon_{13} e^{- \hbar (H_1 \otimes 1 \otimes (H_1 - 1))} (1 \otimes F \otimes E) e^{ \hbar( H_2 \otimes 1 \otimes H_2)}\\
      &=  \Upsilon_{13} e^{- \hbar (H_1 \otimes 1 \otimes (H_1 - 1))} e^{ \hbar( H_2 \otimes 1 \otimes (H_2 +1))} (1 \otimes F \otimes E)\\
      &=  K \otimes F \otimes E.
    \end{aligned}
  \end{equation}
  Therefore
  \begin{equation}
    \label{eq:40}
    \Theta_{13} \Upsilon_{13} \Theta_{23} \Upsilon_{13}^{-1} = (1+ (q-q)^{-1} F \otimes 1 \otimes E) ( 1 + (q-q)^{-1} K \otimes F \otimes E)
  \end{equation}
  coincides with \eqref{eq:38} since $E^2=0$.
\end{proof}

\providecommand{\bysame}{\leavevmode\hbox to3em{\hrulefill}\thinspace}
\providecommand{\MR}{\relax\ifhmode\unskip\space\fi MR }
\providecommand{\MRhref}[2]{%
  \href{http://www.ams.org/mathscinet-getitem?mr=#1}{#2}
}
\providecommand{\href}[2]{#2}

\end{document}